\documentclass[12pt]{amsart}

\oddsidemargin=\dimexpr(\paperwidth-\textwidth)/2-1in\relax
\evensidemargin=\oddsidemargin

\usepackage{epsfig}

\usepackage{amsfonts}
\usepackage{amssymb}
\usepackage{amsmath}
\usepackage{amsthm}
\usepackage{amscd}
\usepackage{amstext}
\usepackage{latexsym}

\usepackage{array}
\usepackage{graphicx, subcaption}
\usepackage{todonotes}

\usepackage{enumerate}

\usepackage{tikz-cd}

\usepackage[all]{xy}

 \numberwithin{equation}{section}

 \usepackage[backref=false,bookmarks=false]{hyperref} %%%%%

\newcommand{\ds}{\displaystyle}

\newcommand{\ZZ}{\mathbb{Z}}
\newcommand{\CC}{\mathbb{C}}
\newcommand{\QQ}{\mathbb{Q}}
\newcommand{\RR}{\mathbb{R}}

\newcommand{\PP}{\mathbf{P}}

\newcommand{\JJ}{\mathcal{J}}

\newcommand{\OO}{\mathcal{O}}
\newcommand{\al}{\alpha}

\newcommand{\mfa}{\mathfrak{a}}
\newcommand{\mfb}{\mathfrak{b}}

\newcommand{\mfm}{\mathfrak{m}}

\newcommand{\mfab}{\mfa_\bullet}
\newcommand{\mfbb}{\mfb_\bullet}

\newcommand{\ep}{\epsilon}

\newcommand{\qa}{\quad}

\newcommand{\vp}{\varphi}

\newcommand{\lan}{\langle}
\newcommand{\ran}{\rangle}

\newcommand{\noi}{\noindent}

\providecommand{\abs}[1]{ |#1|}

\theoremstyle{plain}

\newtheorem{theorem}{Theorem}[section]

\newtheorem{defn}[theorem]{Definition}
\newtheorem{thm}[theorem]{Theorem}
\newtheorem{lem}[theorem]{Lemma}
\newtheorem{lemma}[theorem]{Lemma}
\newtheorem{corollary}[theorem]{Corollary}
\newtheorem{proposition}[theorem]{Proposition}
\newtheorem{definition}[theorem]{Definition}

 \newtheorem{example}[theorem]{\textnormal{\textbf{Example}}}

\theoremstyle{remark}

\newtheorem{remark1}[theorem]{Remark}

\newtheorem{rem}[theorem]{Remark}      
\newtheorem{assumption}[theorem]{Assumption}

\DeclareMathOperator{\loc}{loc}

\DeclareMathOperator{\Spec}{Spec}

\DeclareMathOperator{\QM}{QM}

\DeclareMathOperator{\Div}{Div}

\let \ringaccent=\r  %%% \r  shorthand for 'ring accent'

\renewcommand{\r}[0]{{\mathbb R}}

\newcommand{\q}[0]{{\mathbb Q}}
\newcommand{\map}[0]{\dasharrow}

\newcommand{\spec}[0]{\operatorname{Spec}}

\newcommand{\supp}[0]{\operatorname{Supp}}

\newcommand{\rup}[1]{\lceil{#1}\rceil}

\newcommand{\ord}[0]{\operatorname{ord}}

\newcommand{\tsum}[0]{\textstyle{\sum}}

\newcommand{\coeff}[0]{\operatorname{coeff}} 
\newcommand{\mld}[0]{\operatorname{mld}}

\def\loccoh#1.#2.#3.#4.{H^{#1}_{#2}(#3,#4)}

\DeclareMathAlphabet{\mathchanc}{OT1}{pzc}%
                                {m}{it}

\usepackage[all]{xy}\xyoption{dvips}

\newcommand{\lcg}[0]{\operatorname{lcg}}

\newcommand{\lct}[0]{\operatorname{lct}}

\begin{document}

\title[Log canonical singularities of plurisubharmonic functions]{Log canonical singularities of \\ plurisubharmonic functions}

\author{Dano Kim and J\'anos Koll\'ar}

%\date{}

%\date{\today}

\begin{abstract}
  A plurisubharmonic weight is  \emph{log canonical} if it is at the
critical point of turning non-integrable. Given a log canonical
plurisubharmonic weight, we show that locally there always exists a log
canonical `holomorphic' weight having the same non-integrable locus. The
proof uses the theory of  quasimonomial valuations developed by
Jonsson-Musta\c{t}ă and  Xu, as well as the minimal model program over
complex analytic spaces due to Fujino and Lyu-Murayama. As a consequence,
we obtain that the non-integrable locus is seminormal.
  \end{abstract}
\maketitle

\setcounter{tocdepth}{1}
\tableofcontents

%%%%%%%%%%%%%%%%%%%%%%%

\section{Introduction}

Let $\vp$ be a plurisubharmonic (or \emph{psh}) function on a complex manifold $X$. 
Integrability and non-integrability of singular weights of the form $e^{-2\vp}$ have been of fundamental  importance in several complex variables ever since the work of H\"ormander, Bombieri, Skoda and others on $L^2$ estimates for the $\overline{\partial}$ operator, cf.\ \cite{Ho65, B70, Sk72}. The algebraic counterparts of these singularities also have played a  central role in higher dimensional algebraic geometry and the minimal model program, cf.\ \cite{Ko97, KM, La04} and the references therein. 

Let $\JJ(\vp)$ be the multiplier ideal of $\vp$. A lot of attention had been drawn to Demailly's strong openness conjecture regarding the behavior of $\JJ((1+\ep)\vp)$ as $\ep \to 0$. An interesting algebraic approach was formulated by Jonsson and Musta\c{t}ă~\cite{JM14} before  an analytic proof using $L^2$ estimates for $\overline{\partial}$ was finally given to the conjecture by Guan and Zhou~\cite{GZ15}. 
 
 In this paper, we answer a fundamental question on plurisubharmonic singularities combining these algebraic and analytic methods, together with using the minimal model program. 
 
 Recall that the non-integrable locus $W(\vp) \subset X$ of $\vp$ is the subset of the points where $e^{-2\vp}$ is not locally integrable. 
  Bombieri~\cite{B70} showed the initial result that $W(\vp)$ is closed and contained in an analytic hypersurface~\cite[Theorem 2, Addendum]{B70}. 
  Demailly~\cite[VIII (7.7)]{DX}, \cite[Lem. 4.4]{D93} proved that $W(\vp)$ is a complex analytic subspace defined by a natural coherent ideal sheaf called the {\it multiplier ideal} $\JJ(\vp)$ which is defined by germs of holomorphic functions $f$ such that $\abs{f}^2 e^{-2\vp}$ is locally integrable. 
The  ideals $\JJ(c\vp)$, $c >0$, play a key role in many questions of complex analysis and geometry: see \cite{D11, D13} for  introductions.

The situation is well understood  if  $\vp$  has {\it analytic singularities}, see  Definition~\ref{ansing}.  For example in the divisorial case, this is when  locally
$\vp = \sum^{m}_{i=1} c_i \log \abs{g_i} + u$,  where the $c_i \ge 0$ are real, the  $g_i$ are  holomorphic, and $u$ is bounded. In algebraic and analytic geometry, these are frequently treated as $\RR$-divisors  $\Delta:=\sum^{m}_{i=1} c_iD_i$, where $D_i:=(g_i=0)$.
For divisors, one can use Hironaka's resolution of singularities
to reduce many questions to the special case
$\sum^{m}_{i=1} c_i \log \abs{z_i}$, where the $z_i$ are local coordinates.

For general plurisubharmonic functions $\vp$, such methods are not available, and many new phenomena appear. For example, $\vp$ may be equal to $-\infty$ on a dense subset of the domain, or  the jumping numbers of multiplier ideals $\JJ(c \vp)$ may have accumulation points, cf.  \cite{GL20, KS20, Se21}.

\subsection{Main results} {\ }

In many applications the key question is  to understand the transition from 
integrability to  non-integrability.

Let $\vp$ be a \emph{quasi-plurisubharmonic} (or \emph{quasi-psh}) function  on $X$, i.e. a function $\vp: X \to \RR \cup \{ -\infty \}$ that is locally equal to the sum of a psh function and a smooth real valued function.

Following  algebraic geometry  terminology, we  say that a pair $(X, \vp)$ is {\it klt} at $p \in X$ if $\JJ(\vp)=\OO_X$ at $p$, that is, if $e^{-2\vp}$ is locally integrable at $p$. \footnote{Our convention is to take $e^{-2\vp}$ instead of $e^{-\vp}$.}   We say that $(X, \vp)$ is {\it lc} (or {\it log canonical}) at $p$ if $(X, (1-\delta)\vp)$ is klt for every small $\delta > 0$.

Our main theorem says that, unexpectedly, the multiplier ideals of log canonical 
plurisubharmonic functions are the same as the multiplier ideals of log canonical divisors.

\begin{theorem} \label{main}

  Let $X$ be a complex manifold and $\vp$ a quasi-plurisubharmonic function on $X$. Suppose that $(X, \vp)$ is log canonical at every point of $X$. Then every point $p$ of $X$ has a Stein open neighborhood $V \subset X$ with holomorphic functions $g_1, \ldots, g_m$ on $V$ and real constants $c_1, \ldots, c_m > 0$, such that
  \begin{itemize}  
  \item
 $(X, \psi:=\sum^{m}_{i=1} c_i  \log |g_i|)$ is log canonical at every point of $V$, and
 \item
 $\JJ(\vp) = \JJ(\psi)$ on $V$. 
 \end{itemize}

\end{theorem}

{\it Sketch of proof.}
 It suffices to consider $p$ in $W(\vp)$, otherwise one can take $\psi = 0$. 
A natural attempt could be to use the Demailly approximations $\vp_m$   of $\vp$ as in Definition~\ref{approx.def.jk}.
The $\vp_m$ has analytic singularities, and
$ \JJ \left( (1+\tfrac{1}{m}) \vp_m \right) = \JJ (\vp)$ by 
the strong openness theorem of \cite{GZ15}. However, $(1+\tfrac{1}{m}) \vp_m$ is not log canonical in general. One can  scale to arrange that
 $(1-\epsilon_m)\vp_m$ be
log canonical, but scaling may change the  multiplier ideal.   We go around these problems as follows.

First,  we prove in  Theorem~\ref{logdis}  that $\JJ(\vp)$  can be detected using exceptional divisors (or \emph{divisorial valuations}) with small log discrepancies, just as in the algebraic case (see Sections~\ref{div.val.ss.jk}--\ref{disc.ss.jk} for the notions of divisorial valuation and log discrepancy). 

Demailly approximation and scaling then produces   $(1-\epsilon_m)\vp_m$
and a corresponding 
$\RR$-divisor, say $(1-\epsilon_m)\Delta_m \ge 0$, such that the pair
$\bigl(X, (1-\epsilon_m)\Delta_m\bigr)$ is log canonical and
 has very small minimal log discrepancy along the  zero locus $W(\vp)$ of $\JJ (\vp)$.

 A carefully chosen minimal model program then yields a (non-unique) $\RR$-divisor
 $\Theta_m \ge 0$ such that $\bigl(X, \Theta_m\bigr)$ is log canonical and
 $\JJ (\Theta_m)=\JJ (\vp)$. 
This step is similar to   \cite{Ko14, KK22}, where it was shown that  if  the minimal log discrepancy along a subvariety $Z$ is sufficiently small, then $Z$ behaves like a log canonical center.

In general we cannot choose  $\Theta_m\geq  (1-\epsilon_m)\Delta_m$ which means that  in Theorem~\ref{main}, $\psi$ is not guaranteed to be `more singular' than $\vp$. On the other hand, it is known from \cite{G17} that in general $\psi$ cannot be `less singular' than $\vp$ (i.e. in the sense of $\psi \ge \vp + O(1)$). Perhaps such lack of monotonicity  explains why Theorem~\ref{main} was not expected earlier.  
 \medskip

 The non-integrable locus of a pair $(X, \vp)$, i.e. the complex analytic subspace of $X$ defined by the multiplier ideal sheaf, $W(\vp):=\spec_X \OO_X/\JJ(\vp)$ will be also called the \emph{non-klt locus} of $(X, \vp)$. 
 
As a consequence of
Theorem~\ref{main}, we obtain 
 that the Ambro-Fujino seminormality theorems \cite{A03, Fuj09, fuj-book} continue to hold for
 plurisubharmonic functions.

\begin{corollary}\label{semi}
Let $X$ be a complex manifold and $\vp$ a quasi-plurisubharmonic function on $X$ such that $(X, \vp)$ is log canonical at every point of $X$.
 
(1) Then the non-integrable locus   $W(\vp)$ is  reduced and seminormal.

(2) More generally, any subset of $X$ obtained from  $W(\vp)$ by repeatedly taking irreducible components, intersections and unions is reduced and seminormal.
\end{corollary}

 The part (1) answers Question 1.1 of \cite{Li20} (cf. \cite{GL16}). 
 
 See Definition~\ref{sn.defn.jk} for the equivalent notions of 
  seminormality and weak normality.  Note that in general, an irreducible component of a seminormal complex analytic space need not be seminormal, cf. \cite[(7.9.3), Section 10.2]{Ko13}. Corollary~\ref{semi} produces many examples of complex analytic subspaces of $X$ that cannot be realized as the non-integrable  locus of log canonical plurisubharmonic functions. \footnote{This can be viewed as giving answers to the log canonical version of the question: \emph{Which ideals can be realized as multiplier ideals?}, cf. \cite{LL06}, \cite{Ki10}.}

 Indeed, the non-integrable  locus of a log canonical plurisubharmonic function $\vp$  occupies a rather special place among many other subsets of $X$ one can attach to a given plurisubharmonic function (cf. Section 2.1) in that we can expect such `regular' properties when $e^{-2\vp}$ is at the critical point of turning non-integrable. (For example, one would not expect some nice properties for the sublevel set of $\lct \le 2$, while the non-integrable locus is equal to the sublevel set of $\lct \le 1$.) This aspect has been well understood from algebraic geometry when $\vp$ has (divisorial) analytic singularities. Corollary~\ref{semi} confirms that it continues to hold for general plurisubharmonic $\vp$.

\subsection{Log canonical places} {\ }

 In the proof of Theorem~\ref{main},  a key role is played by the next result which can be viewed as  strengthening \cite[Thm.B.5]{BBJ21}.

 \begin{theorem}[=Theorem~\ref{converge2}]\label{logdis}
   Let $(X, \vp)$ be a log canonical pair where $X$ is a complex manifold and $\vp$ a quasi-psh function on $X$.  Let $Z \subset X$ be an irreducible component of the zero set  $W(\vp)$.
Then there is a  sequence of prime divisors $G_j$   over $X$ ($j \ge 1$) whose centers on $X$ are equal to $Z$,   such that  the log discrepancies $A(G_j, X, \vp)$ converge to $0$ as $j \to \infty$.  
\end{theorem} 

A prime divisor over $X$ is also called a geometric divisorial valuation in this paper (see Section 3). 
 Previously it was known that quotients  $A(G_j, X, \vp)/G_j(\vp)$
 converge to $0$.  In most cases the $G_j(\vp)$ diverge,  and
 Theorem~\ref{logdis} seems the optimal result. What is crucial in the proof of Theorem~\ref{logdis} is that we are able to first obtain the `limit' of these $G_j$'s as a quasimonomial valuation $v$ with $A(v, X, \vp)=0$ (Theorem~\ref{computer}) thanks to the work of  Jonsson-Musta\c{t}ă \cite{JM12, JM14} and Xu \cite{X20} and then use $v$ to obtain $G_j$'s.

Such a  quasimonomial valuation plays the role of pinpointing  where on a modification of $X$, non-integrability occurs. This  generalizes its classical divisorial counterparts called log canonical places (cf. \cite{Ka97}, \cite{Sh92}) for a log canonical center when $\vp$ has analytic singularities. 
Theorem~\ref{logdis} suggests the following generalization of the notion of a log canonical center.

\begin{definition}\label{lccenter} Let $(X, \vp)$ be a log canonical pair where $X$ is a complex manifold and $\vp$ a quasi-psh function on $X$. An irreducible closed analytic subset $Z \subset X$ is a 
\emph{log canonical center} of $\vp$ if 
 there is a  sequence of prime divisors $E_j$  over $X$ whose centers on $X$ are equal to $Z$,  such that  the log discrepancies $A(E_j, X, \vp)$ converge to $0$. 
\end{definition}

Note that this definition bypasses the need to show the existence of quasimonomial log canonical places in general (i.e. beyond those already given by Theorem~\ref{computer} for maximal log canonical centers). Also it is worth noting that in Theorem~\ref{main}, $(X, \vp)$ and $(X, \psi)$ may have different non-maximal log canonical centers  (in the sense of Definition~\ref{lccenter}).

In this paper, our use of Definition~\ref{lccenter} mainly lies in formulating the following result on the minimal log discrepancies (Definition~\ref{mld}). In the case when $Z$ is a maximal log canonical center, i.e. an irreducible component of the non-klt locus, it is from Theorem~\ref{logdis} that the condition of Definition~\ref{lccenter} holds for $Z$. We will use that case in the proof of Theorem~\ref{main}.

\begin{corollary}\label{four} 
Let $(X, \vp)$ be a log canonical pair where $X$ is a complex manifold and $\vp$ a quasi-psh function on $X$.  Let $Z \subset X$ be a log canonical center of $\vp$ (in the sense of Definition~\ref{lccenter}).  Let $\vp_m$ be the  $m$-th Demailly approximation of $\vp$. Then we have
$
\lim_{m \to \infty} \mld(Z, X,\vp_m)=0.
$
\end{corollary}

\subsection{Comparison with previous related results} {\ }

Finally we say a word on valuations and the valuative theory of plurisubharmonic singularities due to Favre-Jonsson in dimension $2$ and to Boucksom-Favre-Jonsson in general dimension  (cf. \cite{FJ05I}, \cite{FJ05}, \cite{BFJ08}, \cite{BFJ14}, \cite{JM12}, \cite{JM14} and others), which play important roles in this paper. We use `geometric' versions of divisorial and quasimonomial valuations (following \cite{BBJ21}, \cite{B21}, \cite{BFJ08}, see Sections 3 and 4) which are more suitable for our setting  of a complex manifold. On the other hand, when $X$ is a smooth complex projective variety,  the geometric divisorial valuations appearing  in Theorem~\ref{logdis} and the geometric quasimonomial valuation in Theorem~\ref{computer} can be taken as the usual algebraic valuations of the function field of $X$, see Corollary~\ref{algeq}.

\begin{remark1} 
It is reasonable to expect our results to  extend to
  plurisubharmonic functions on a complex analytic space $X$ with log terminal singularities. For that purpose, the existence of a Demailly-like approximation needs to be first developed, cf. \cite{B21}. On the other hand, when $X$ has log canonical singularities, 
the proofs in Section~\ref{sec.6.jk} have several (presumably technical) problems. \end{remark1}

\begin{remark1}
For  algebraic log canonical pairs $(X, \Delta)$, \cite{KK10, KK20} prove that the union of an arbitrary collection of log canonical centers   has  Du~Bois singularities. It is quite likely that the same holds in the analytic setting of Theorem~\ref{main} as well, but, as O. Fujino pointed out,  the
current proofs for the algebraic case use some projective methods.
\end{remark1}

\noi \textbf{Acknowledgment.}
We would like to thank Chenyang Xu for informing us of an argument of \cite{LXZ21}. We would like to thank Sébastien Boucksom for helpful discussions on valuations and Mattias Jonsson for answering our questions. 
 DK was supported by Basic Science Research Program through NRF Korea funded by the Ministry of Education (2018R1D1A1B07049683).
Partial  financial support  to JK   was provided  by  the NSF under grant number
DMS-1901855.

\section{Preliminaries}

For some detailed introduction to plurisubharmonic functions, we refer to  \cite{DX, Ks93, B21}.

\subsection{Plurisubharmonic functions and their singularities} {\ }

 Let $U \subset \CC^n$ be an open subset. Let $\vp: U \to \RR \cup \{ -\infty \}$ be an upper-semicontinuous function that is not identically equal to $-\infty$. It is called \emph{plurisubharmonic} (or \emph{psh}) if, for every complex line $L \subset U$, the restriction $\vp|_L$ is either subharmonic or $\equiv -\infty$. 
 
A psh function on a complex manifold $X$ is also well-defined. A function on a complex manifold with values in $\RR \cup \{ -\infty \}$ is \emph{quasi-plurisubharmonic} (or \emph{quasi-psh}) if it is locally equal to the sum of a psh function and a smooth real valued function.

\begin{definition} \cite{D11} \label{ansing}
A quasi-psh function $\vp$ on a complex manifold $X$ is said to have {\emph{analytic singularities}} of type $\mfa^c$ if it is locally of the form $\vp = c \log \sum^{m}_{i=1} \abs{g_i} + u$ where $c \ge 0$ is real, $u$ bounded and $g_1, \ldots, g_m$ are local holomorphic functions generating a (global) coherent ideal sheaf $\mfa \subset \OO_X$. 

 We will say that $\vp$ has \textbf{\emph{analytic singularities}} if it is locally the sum of a finite number (say $k$) of psh functions with analytic singularities of types $\mfa_1^{c_1}, \ldots, \mfa_k^{c_k}$ where  $c_1, \ldots, c_k \ge 0$ are real numbers and $\mfa_1, \ldots, \mfa_k$ are coherent ideal sheaves on $X$. 
\end{definition}

  We may sometimes informally say  that a psh function is \emph{general psh} when it does not necessarily  have analytic singularities. 
  
   \begin{remark1}
 
 Note that Definition~\ref{ansing} is slightly more general than the original one where $k=1$, cf. \cite{D11}. The use of this generality is justified due to the key property of having a log resolution, which makes the usual properties to hold.  We need this generality in this paper since we will use effective $\RR$-divisors. 
 \end{remark1}
 
  Let $D$ be an effective $\RR$-divisor (i.e. a locally finite sum of prime divisors) on $X$. As a special case of Definition~\ref{ansing}, let us say that a quasi-psh function $\vp$ has \emph{algebraic poles \footnote{This terminology is slightly ad-hoc with respect to Definition~\ref{ansing},  but of course here we mean `algebraic = analytic', 'poles = singularities'.} along} $D$ if we have locally 
 
 \begin{equation}\label{fsum}
 \vp = \sum^m_{i=1} a_i \log \abs{g_i}  + u 
 \end{equation}
  where $u$ is bounded and $\sum^m_{i=1} a_i \Div(g_i) $ is a local expression of the divisor $D$.

  A \emph{log resolution} for a quasi-psh function $\vp$ with analytic singularities   is a proper modification   $f: X' \to X$ such that the pullback $f^* \vp$ has algebraic poles along a divisor $G$ on $X'$ and $G+F$ is an snc (i.e. simple normal crossing) divisor where $F$ is the exceptional divisor of $f$. Such a log resolution exists since principalizations of the ideal sheaves $\mfa_1, \ldots, \mfa_k$ exist by Hironaka~\cite{H64}.

\begin{remark1}
We note that significant prior work on psh singularities in dimension $2$ was conducted by Favre and Jonsson~\cite{FJ04}, \cite{FJ05}, \cite{FJ05I}. In particular, in \cite[Sec. 7]{FJ05I}, a weak form of resolution was obtained (cf. \cite{G05}). 
\end{remark1}

When the previous sum \eqref{fsum} is not locally finite, one can get examples of $\vp$ with non-analytic singularities. The following are examples on $\CC$ (which can be also viewed on $\CC^n$ depending only on one variable).  
 
 \begin{example}
 
 Let $\vp(z) = \log \abs{z} + \sum_{k \ge 2} 2^{-k} \log \abs{z- \frac{1}{k}}$. This does not have analytic singularities since the pole set $\vp^{-1} (-\infty)$ is not an analytic subset, cf. \cite[Example 3.1]{KR18}. 
 \end{example}
 
 \begin{example} \cite[(2.5.4)]{Ra95}
 Let $w_j$ be a countable dense subset of a given compact subsdet $K \subset \CC$ with no isolated points. Let $a_j > 0$ be such that $\sum_{j \ge 1} a_j < \infty$. Then $\vp(z) = \sum_{j \ge 1} a_j \log \abs{z - w_j}$ is subharmonic on $\CC$, not identically $-\infty$. However, $\vp = -\infty$ holds on an uncountable dense subset of $K$. 
 
 \end{example}

For another example of a psh function $\vp$ with a dense pole set $\vp^{-1} (-\infty)$, see \cite[Example 2.1]{ACH15}.

Now we recall some fundamental invariants of singularities of psh functions and some of their associated subsets.

 The {\it multiplier ideal} $\JJ(\vp)$ is the ideal sheaf of germs of holomorphic functions $f$ such that $\abs{f}^2 e^{-2\vp}$ is locally integrable. It is coherent, cf.  \cite[Lem. 4.4]{D93}. 
The \emph{log canonical threshold} of $\vp$ at $x$ is defined to be 
   $  \lct_x (\vp) := \sup \{ c \ge 0 : \JJ(c \vp)_x = \OO_{X, x} \} $. 
   
   Given $\vp$, we can compare some subsets of $X$ defined in terms of these invariants. 

\begin{itemize}
\item
  Recall the non-integrable locus (or the non-klt locus) $$W(\vp):=\spec_X \OO_X/\JJ(\vp),$$ the complex analytic subspace defined by $\JJ(\vp)$.  Note that as a set,   $W(\vp)$  is a (typically proper) subset of the pole set
$\vp^{-1} (-\infty)$ (as one checks easily from applying \cite{OT87} to each point).  
 
 \item
 On the other hand,  it is a classical result of Skoda~\cite{Sk72a} that $$E_1 (\vp) \supset W(\vp) \supset E_n (\vp)$$ where $E_c  (\vp):= \{ x \in X : \nu (\vp, x) \ge c \}$ (for $c \ge 0$) are superlevel sets of Lelong numbers, cf. \cite[Lem.5.6]{D11}. (See \eqref{Kisel} for the definition of the Lelong number $\nu(\vp, x)$.) Note that each $E_c$ is a closed analytic subset of $X$ due to the fundamental result of  \cite{Si74} (cf. \cite[Cor.14.3]{D11}). 

\item
 Using log canonical thresholds, define $V_c \subset X$ to be the sublevel set $$V_c :=\{ y \in X : \lct_y (\vp) \le c \} $$ which is a closed analytic subset of $X$ due to \cite[1.4 (1)]{DK01}. Note that $V_1 = W(\vp)$ holds by  the openness theorem first proved by Berndtsson in \cite{Be13} (cf. \cite{Be15}), which says that if $\lct_y (\vp) =1$, then $e^{-2\vp}$ is locally non-integrable at $y$, that is $y \in W(\vp)$.

\end{itemize}

We remark that there have been various previous and ongoing work on the singularities of plurisubharmonic functions: we refer to \cite{Ks93}, \cite{Ks94}, \cite{D92}, \cite{D93},  \cite{DK01}, \cite{FJ05I}, \cite{FJ05}, \cite{Si05}, \cite{BFJ08},  \cite{D11}, \cite{R13}, \cite{JM14},  \cite{GZ15}, \cite{G17}, \cite{GL20},  \cite{KR18}, \cite{KR20}, \cite{Se21}, \cite{DX22}, \cite{X24}, \cite{DX24} only for some of them in relation to this paper.

 \subsection{Demailly approximation of psh functions}{\ }

  A general psh function can be approximated by ones with analytic singularities, due to the following fundamental result of Demailly~\cite{D92}. 
 
 \begin{definition}  \label{approx.def.jk} \cite{D92}
 Let $D \subset \CC^n$ be a bounded pseudoconvex domain.    The $m$-th  Demailly approximation of $\vp$ is $$ \vp_m (z) :=  \tfrac{1}{m} \sup \{ \log \abs{f(z)} : f \in \OO(D), \int_D \abs{f}^2 e^{-2m\vp} dV <1 \} .$$
\end{definition}

 The main properties of $\vp_m$ are summarized as follows. (See \cite[Thm.13.2]{D11} for the full statements.)

 \begin{theorem} \label{approx} \cite{D92}, \cite[Thm.13.2]{D11}
Using the above notation, 
 the  $\vp_m$ are psh functions with analytic singularities. As $m \to \infty$, the  $\vp_m$ converge to $\vp$ pointwise and in the $L^1_{\loc}$ norm. Also at each point, the Lelong number of $\vp_m$ converges to that of $\vp$ as $m \to \infty$.  
 \end{theorem}
 
 For a complex manifold $X$ and a quasi-psh function $\vp$ on $X$, one can also define a Demailly approximation $\vp_m \to \vp$ of quasi-psh functions $\vp_m$ using partitions of unity as in \cite[\S 13]{D11}, gluing local approximants. This is possible since local approximants $\vp_m$ have the same analytic singularity types of $\JJ(m\vp)^{\tfrac{1}{m}}$. 
 
  We  refer to it as `the' Demailly approximation since the interest of this paper lies only in singularities, and the choices of open covers and partitions of unity do not matter for our purpose.  
 
  In general, we have the relation $\vp \le \vp_m + O(1)$ for every $m \ge 1$ and thus $\JJ(\vp) \subset \JJ(\vp_m)$. However we do not have equality $\JJ(\vp) = \JJ(\vp_m)$ in general. 
Instead, a consequence from the strong openness theorem of \cite{GZ15} is as follows, cf.\ \cite{D13}. 

 \begin{lemma}\label{1m} \cite[Remark 3]{D13} 
 Let $\vp$ be a quasi-psh function on a complex manifold. Let $\vp_m$ be  the $m$-th Demailly approximation of $\vp$. When $m$ is sufficiently large, we have 
  \begin{equation}\label{equality}
   \JJ \left( (1+\tfrac{1}{m}) \vp_m \right) = \JJ (\vp) .
   \end{equation}
 
 \end{lemma}
 
  This result says that  for every psh function, its multiplier ideal can be realized as that of a psh function with analytic singularities. 
  
  \begin{remark1}
 Note that this lemma holds despite some related subtleties. First, the functions $\vp_m$ (and their singularities) are not monotone decreasing, cf. \cite{Ki14}.  Also  there could be infinitely many jumping numbers of $\vp$ accumulated near a real number $c > 0$, due to examples in \cite{GL20, KS20, Se21}. 
 
 \end{remark1}

\begin{remark1}\label{lclc}
 If one tries to use Lemma~\ref{1m} for Theorem~\ref{main}, by letting $\psi := (1+\tfrac{1}{m}) \vp_m$, the problem is that when one scales by $c < 1$, so that $(X, c \psi)$ is lc, a priori it may be the case that  $\JJ(c\psi)$ is different from $\JJ(\psi)$. This occurs precisely when there exist jumping numbers in the interval $(c, 1]$.

 \end{remark1}
  
\subsection{Singularities of pairs}  {\ }

 In algebraic geometry, a fundamental role is played by the notion of `singularities of pairs' (cf.\ \cite{Ko97, Ko13}). Also see \cite{F2201} for singularities of pairs in the setting of complex analytic spaces. In this paper, we  use the terminology of pairs in the setting of quasi-psh functions on a complex manifold.

Let $X$ be a complex manifold and $\vp$ a quasi-psh function on $X$. We denote this data as a pair $(X, \vp)$. Of course, this extends the usual pairs (cf. \cite{Ko97}) with a divisor in the following sense. 

\begin{example}\label{2.7.exmp.jk}  Let $g_i$ be holomorphic functions with divisors $D_i:=(g_i=0)$. Let $\Delta = \sum_i a_i D_i$ be a finite sum for some $a_i\geq 0$. Then
  $\vp_{\Delta}:=\sum_i a_i \log\abs{g_i}$ has the same singularities as 
  $\Delta$.  Using a partition of unity, one can get such a $\vp_{\Delta}$
  for any  locally finite sum $\Delta = \sum_{i \in I} a_i D_i$.
  We set  $\JJ(\Delta):=\JJ( \vp_{\Delta})$. 
\end{example}

   It is easy to see that
  $(X, \Delta)$ is klt (resp.\ lc) in the `algebraic sense' (i.e. in the sense of \cite{Ko97} for a complex manifold) iff
  $(X, \vp_{\Delta})$ is klt (resp.\ lc) in the sense given before Theorem~\ref{main}; see for example \cite[Prop.8.3]{Ko13}.
A precise version of this is the following;
 see Definition~\ref{discrep} for the definition of $A(E,X,\vp)$.

\begin{proposition}\label{lc-klt}

Let $\vp$ be a quasi-psh function with analytic singularities on a complex manifold $X$. 

 (1) The pair $(X, \vp)$ is lc  if and only if $A(E, X, \vp) \ge 0$ for every prime divisor $E$ over $X$. 
 
 (2) The pair $(X, \vp)$ is klt  if and only if $A(E, X, \vp) > 0$ for every prime divisor $E$ over $X$.

 \end{proposition}

The equivalence in (1) still holds when $\vp$ is general psh by
Theorem~\ref{logdis}, however (2) does not hold in general, since the pair can be non-klt having a non-divisorial but quasimonomial lc place (cf. \cite{JM12}) as was discussed in the introduction. 

\quad
\\
 
\subsection{Seminormality} {\ }

We recall the notion of seminormality and weak normality introduced in  \cite{AN67}, \cite{AB69}; see also  \cite{GT80, BM19, Ko13}.\footnote{The two notions agree in characteristic 0. Analysts seem to prefer weak normality, algebraists seminormality.}

\begin{definition}\label{sn.defn.jk}
 A reduced complex analytic space $X$ is \emph{\textbf{seminormal}} (or \emph{weakly normal}) if the following condition holds for every open subset $U$ of $X$: 

(*) If a continuous function $g: U \to \CC$ is holomorphic on a dense Zariski open subset of $U$, then $g$ is holomorphic on $U$. 
\end{definition}

Note that a complex algebraic variety $W$ is seminormal as an algebraic variety (cf. \cite{GT80, Ko13}) if and only if the associated complex analytic space $W^h$ is seminormal in the sense of Definition~\ref{sn.defn.jk} by \cite[Cor. 6.13]{GT80}.

For example, the node $(xy =0)$ is seminormal, but the cusp $(x^2 - y^3 =0)$ and the ordinary triple point $(x^3 - y^3 = 0)$ are not seminormal (cf.\ \cite[10.12]{Ko13}).

Next,   we recall the Ambro-Fujino seminormality theorems \cite{A03, Fuj09, fuj-book}
in a form that we need.

\begin{theorem}\label{Ambro2003}

  Let $(X, \Delta)$ be an lc pair where $X$ is a complex analytic space and $\Delta \ge 0$ is an $\RR$-divisor. Let
  $W_\Delta$ be the non-klt locus of $(X, \Delta)$. Then $W_\Delta$ is reduced and seminormal.
  Moreover,  any subscheme obtained from $W_\Delta$ by 
repeatedly taking irreducible components, intersections and unions is also reduced and seminormal.
 
\end{theorem}

We use this result in the proof of Corollary~\ref{semi} when $X$ is a complex manifold. In that case,  $W_\Delta$ is equal to $W(\vp_\Delta)=  \spec_X \OO_X/\JJ(\vp_\Delta)$ where $\vp_\Delta$ is defined as in Example~\ref{2.7.exmp.jk}. 

{\it Comments on the proof.} For the seminormality of  $W_\Delta$, the
short
proof
in \cite[Thm.8.8.1]{Ko07} works without changes, using that the relevant
vanishing theorems hold over complex analytic spaces. The latter are
discussed in
\cite[Sec.5]{F2201}.

The more refined version about  irreducible components, intersections and
unions  is implied by the following two claims:
\begin{enumerate}
   \item every intersection of log canonical centers is also a union of log
canonical centers, and
\item every union (i.e. the union of an arbitrary collection) of log canonical centers is seminormal.
\end{enumerate}
These are both proved in \cite{A03, Fuj09, fuj-book};
see also \cite[Thms.7.5--6]{Ko13} for generalizations. \footnote{In the generalities of $X$ being a complex manifold (to be used in this paper) and a complex analytic space, one may also see the expositions in \cite{Ki21} and in \cite{F22} respectively.}

 \begin{remark1}
 Let $X = \CC^3$. Let $I_r \subset \OO_X$ be the ideal of $r$ very general lines through the origin $0 \in X$. In Theorem 2.2.2 of  \cite{Le08}, it is shown that if $r \le 10$, then $I_r$ can be realized as a multiplier ideal, i.e. there exists a psh function $\vp$ with analytic singularities (actually coming from a divisor) such that $I_r = \JJ(\vp)$. If we further require this $\vp$ to be lc, then  the non-klt locus $W(\vp)$ is seminormal by Corollary~\ref{semi}.  From the seminormality, it follows that $r \le 3$ by \cite[p.308]{Ko13}.  
 \end{remark1}

\section{Divisorial valuations}

On a complex algebraic variety $X$, a divisorial valuation can be understood both algebraically (as a discrete valuation ring on the function field) and geometrically (as a prime divisor over $X$). In particular, a classical method following Zariski  says that these two aspects coincide, cf.\ \cite[Lem.2.45]{KM}. 

 Since we work in the complex manifold setting in this paper, we do not have such coincidence  and we need to define a valuation geometrically (both for divisorial and quasimonomial) following the treatments of \cite{BFJ08}, \cite{BBJ21}, \cite{B21}.  \footnote{On the other hand, one can refer to e.g. \cite{H66} for some existing work in the direction of using algebraic valuations in the setting of complex analytic spaces.}

 \subsection{Divisorial valuations}\label{div.val.ss.jk}  {\ }

 A holomorphic map between complex manifolds $f: X \to Y$ is a \emph{proper modification} (or simply a modification) if it is proper and there exists a nowhere dense, closed analytic subset $B \subset Y$ such that $f: X \setminus f^{-1} (B) \to Y \setminus B$ is biholomorphic, cf.\ \cite[p.126]{DX}, \cite{H62}.

 Let $E$ be a  \emph{prime divisor  over} a complex manifold $X$,  by which we mean a closed connected smooth hypersurface $E \subset X'$ where $f: X' \to X$ is a modification (following the terminology of \cite[B.5]{BBJ21}).  We define the center of $E$ on $X$ to be $c_X (E) := f(E)$. 
 
  Let $x \in c_X (E)$ be an arbitrary point and let $\vp$ be a germ of a psh function at $x$. The divisor $E$ defines a `valuation' $v$ on psh germs by taking $v(\vp)$  to be the generic Lelong number of $f^* \vp$ along $E$, i.e. $v(\vp) := \inf_{y \in E} \nu (f^* \vp, y)$ where $\nu$ denotes the Lelong number  (cf. \cite{Si74}, \cite[(2.17)]{D11}). 
  
  In particular, $v$ is extending the valuation $\ord_E$ on the local ring $\OO_{X, x}$ of germs of holomorphic functions at $x$, obtained by $\ord_E (g) := v( \log \abs{g})$. In this sense, we call $E$  a \emph{geometric divisorial valuation}, which we may also denote by $v = \ord_E$ (as abuse of notation).  More generally, a geometric divisorial valuation is of the form $v = c \ord_E$ for some $E$ and  $c > 0$. 
 
 Prime divisors $E_1, E_2$ appearing in different modifications $f_1 : X_1 \to X$ and $f_2 : X_2 \to X$ respectively, are \emph{equivalent} when their strict transforms on a common modification $X'$ over $X_1, X_2$ coincide. In that case, $E_1$ and $E_2$ define the same geometric divisorial valuations in the above sense.

\subsection{Log discrepancies}\label{disc.ss.jk}  {\ }
 
   For a geometric divisorial valuation of the form $v = \ord_E$,  we define the \emph{log discrepancy of the valuation} (following the terminology of \cite{BFJ08, JM12})  as
\begin{equation}\label{avv}
  A(v) := 1+ v (K_{X'/X}),
\end{equation}  
\noi   where $K_{X'/X}$ is the relative canonical divisor. More generally, for a divisorial valuation $v = c \ord_E$ with $c >0$, we define $A(v) := c A(\ord_E)$. 

Next, we define the log discrepancy of a valuation with respect to a pair: 
  
  \begin{definition}\label{discrep}
 For a psh function $\vp$ on $X$, we  define the {\bf{log discrepancy}} of $(X, \vp)$ along $E$ by
  $$A(E, X, \vp) := A(v, X, \vp) := A (v)  - v (\vp)$$  where $v = \ord_E$ and $v (\vp)$ is the generic Lelong number (cf. \cite[(2.17)]{D11}, \cite[B.6.]{BBJ21}) of the pullback $f^* \vp$ along $E \subset X'$.
 \end{definition} 
 
     Note that $a(E, X, \vp) := A(E, X, \vp) -1$ is the   \emph{discrepancy}, as in  \cite{KM}. For our purpose   the log discrepancy is convenient, since it is  homogeneous, i.e. $A(cv, X, \vp) = c A(v, X, \vp)$ for real $c \ge 0$. In particular, the condition $A(v, X, \vp) \ge 0$ is preserved by scaling.

   In this paper, we will  use the following  notion only when $\vp$ has analytic singularities. 

 \begin{definition}\label{mld}\cite{A99}, \cite[Def.2.9]{Ko13}
 Let $(X, \vp)$ be an lc pair. Let $Z$ be a closed irreducible analytic subset of $X$. The \textbf{\emph{minimal log discrepancy}} of $Z$ with respect to $(X, \vp)$ is defined as 
  $$\mld (Z, X, \vp) = \inf_E A(E, X,  \vp) $$ where the infimum is taken over all prime divisors $E$ over $X$, whose center on $X$ is  $Z$. 
If  $W$ is a closed analytic subset of $X$, we define $$\mld (W, X, \Delta) := \sup_{i} \mld (W_i , X, \Delta)$$ where the $W_i$'s are irreducible components of $W$. 
\end{definition}

 \begin{example}
 
 Let $(X, \vp) = (\CC^2, \frac{5}{6} \log \abs{z^2 - w^3})$ which is an lc pair. When $Z$ is the curve $(z^2 - w^3 =0)$, $\mld (Z, X, \vp) = \frac{1}{6}$, see \cite[Example 5, (1)]{Ko14}. The pole set of $\vp$ is $Z$ (which is not seminormal) and $W(\vp) = \{ (0,0) \}$.

 \end{example}

 \subsection{Valuative characterization of multiplier ideals}  {\ }

  The following valuative characterizations of multiplier ideals and log canonical thresholds  are originally due to Boucksom-Favre-Jonsson~\cite{BFJ08} combined with the use of the strong openness theorem of Guan-Zhou~\cite{GZ15} (cf. the earlier work \cite{FJ05} when the dimension is $2$).

\begin{theorem} \cite{BFJ08}, \cite[Thm.B.5]{BBJ21}\label{vc}
 Let $X$ be a complex manifold. Let $\vp$ be a psh function on $X$. For every $x \in X$, we have 
 $$ \JJ (\vp)_x  = \{ f \in \OO_{X, x} : { \exists  \ep > 0, }\; { v(f) + A(v)}{} \ge (1+ \ep) v(\vp) \text{ for every } v  \} $$
\noi where $v$ ranges over all geometric divisorial valuations whose center $c_X (v)$ is equal to $\{ x \}$. 

\end{theorem}

 \begin{corollary} \cite{BFJ08}, \cite{BBJ21} \label{char} 
   The \emph{log canonical threshold} 
   $$  \lct_x (\vp) := \sup \{ c \ge 0 : \JJ(c \vp)_x = \OO_{X, x} \} $$
   equals  $\inf_{v \in T_x} \bigl({A(v)}/{v(\vp)}\bigr)$  where $T_x$ is the set of all geometric divisorial valuations whose center $c_X (v)$ is equal to $\{ x \}$.

 \end{corollary}

 \noi (For a related result in dimension $2$, see \cite{V76}, \cite[Thm. 6.40]{KSC04}.)
 In view of these results, we may define two quasi-psh functions $\vp$ and $\psi$ on $X$ to be \emph{valuatively equivalent} if $v(\vp) = v(\psi)$ for every geometric divisorial valuation $v$ on $X$. Theorem~\ref{vc} says that in this case, the equality of multiplier ideals $\JJ(m\vp) = \JJ(m \psi)$ holds for all real $m \ge 0$. We note the following fact which points out that 
 in the conclusion of Theorem~\ref{main}, in general there may be some $m > 1$ for which $\JJ(m\vp) \neq \JJ(m\psi)$. 
 
 \begin{proposition}
 
 There exists a (germ of) psh function $\vp$ that is not valuatively equivalent to any psh function with analytic singularities.  
 
 \end{proposition}
 
 \begin{proof}
 
 One way to see this is when we take $\vp$ to be one with an accumulation (i.e. a cluster) point of jumping numbers at a point $p \in \CC^n$ (see  e.g. \cite[Def. 3.1]{KS20}), \cite{GL20}, \cite{KS20}, \cite{Se21}. When $\psi$ is a psh function with analytic singularities on $X$ such that $\psi$ and $\vp$ are valuatively equivalent, $\vp$ and $\psi$ have the same jumping numbers since $\JJ(m \vp) = \JJ( m \psi)$ for every real $m \ge 0$. Since $\psi$ cannot have an accumulation point of jumping numbers due to the existence of a log resolution, we get contradiction. 
  \end{proof}

  \begin{remark1}
  In the case of dimension $2$, Favre and Jonsson~\cite[Theorem 8.2]{FJ04}  obtained a characterization of psh functions in $(\CC^2,0)$ that are valuatively equivalent to ones with analytic singularities. The referee brought this result to our attention.
  \end{remark1}

 \section{Quasimonomial valuations} 
 
 We introduce quasimonomial valuations. 
  We  follow the exposition of \cite{JM14, B21, JM12}; see also  \cite{BFJ08, BFJ14}.

\subsection{Monomial valuations in terms of Kiselman numbers} {\ }
 
  We first recall the following generalization of Lelong numbers, cf.\ \cite{Ks93, JM14}. 
Let $\vp$ be a psh function near the origin $0 \in \CC^n$ with  coordinates $(z_1, \ldots, z_n)$. 
Let $\al =(\al_1, \ldots, \al_n) \in \RR^n_{\ge 0}$. 
The \emph{Kiselman number} of $\vp$ at $0 \in \CC^n$ with the weight $\al$  is defined by 

\begin{equation}\label{Kisel}
 v_\al (\vp) := \sup \Bigl\{ t \ge 0 : \vp \le t \log \Bigl( \max_{1 \le j \le n, \al_j > 0} \abs{z_j}^{\frac{1}{\al_j}} \Bigr) + O(1) \Bigr\}. 
\end{equation}

The case when every $\al_j = 1$ corresponds to the usual Lelong number $\nu(\vp, 0)$.  It is known that the sup in \eqref{Kisel} can be replaced by max.  When $\vp = \log \abs{f}$ for a holomorphic function germ $f \in \OO_{\CC^n, 0}$, $v_\al (f) := v_\al (\log \abs{f})$ is indeed a valuation on the local ring $\OO_{\CC^n, 0}$  in that $v_\al (fg) = v_\al (f) + v_\al (g)$ and $v_\al (f+ g) \ge \min (v_\al (f), v_\al (g))$.

 Moreover, for a local power series $\ds f = \sum_{\beta \in \ZZ_{\ge 0}^n} c_{\beta} z^{\beta},$ we have $\ds v_\al (f) = \min_{c_{\beta} \neq 0} \; \lan \beta, \al \ran$ where $\lan \beta, \al \ran := \beta_1 \al_1 + \ldots + \beta_n \al_n$ is equal to $v_\al (z^{\beta})$, a `monomial valuation' of $z^{\beta}$.

\subsection{Quasimonomial valuations}  {\ }

Now let us globalize the setting with an snc divisor from local coordinates, on a proper modification.    Let $X$ be a complex manifold of dimension $n \ge 1$. Let $\mu: Y \to X$ be a proper modification from another complex manifold. Let $D = \sum_{i \in I} D_i$ be an snc divisor on $Y$ where $I$ is a finite index set.\footnote{A finite $I$ is sufficient for us since our interest is local on $X$ (or when $X$ is compact).}  

Let us define valuations on $X$ described in terms of $(Y, D)$. 
Let $J = \{ j_1, \ldots, j_m \} \subset I$ be a subset of indices with $m := \abs{J} \le n$ such that  the stratum given by  

\begin{equation}\label{DJ}
Z =   D_{j_1} \cap \ldots \cap D_{j_m} 
\end{equation}
 is a (nonempty) connected submanifold of codimension $m \ge 1$. 
Let $(z_1, \ldots, z_n)$ be local coordinates centered at an arbitrary point $p \in Z$ such that $D_{j_i}$ is locally defined by $z_i = 0$ for $1 \le i \le m$.

Let  $\Omega \subset X$ be a neighborhood of $p$. Let $\vp$ be a psh function on $\Omega$. We take the Kiselman number $v_\al (\vp, x)$ (for $x \in Z$) as in \eqref{Kisel} taking $x$ in the place of $0$ and keeping $\al_{m+1} = \ldots = \al_{n} = 0$.

The important properties of Kiselman numbers given by \cite[Lem. 3.1]{JM14} says that the Kiselman number $v_\al (\vp, x)$ is constant for $x \in Z \cap \Omega$ and that $v_\al (\vp, x)$ depends only on the divisor $D_{j_1} +\cdots+ D_{j_m}$, not on the particular choice of analytic coordinates adapted to it. 
Due to these properties, we shall denote this Kiselman number as $\tau_{Z,D, \al} (\vp)$.

Now let $x$ be an arbitrary point of the image $\mu(Z)$ under $\mu: Y \to X$. As in the case of a geometric divisorial valuation, $\tau_{Z, D, \al}$ defines a `valuation' $v$ on psh germs $\psi$ at $x$ by taking $v(\psi) := \tau_{Z, D, \al} (\mu^* \psi)$. This extends the corresponding valuation on the local ring $\OO_{X, x}$ obtained by taking $v(\log \abs{g})$ for $g \in \OO_{X, x}$. 

In this sense, we will call $ \tau_{Z, D, \al}$ (or the induced $v$) as a  \emph{{geometric quasimonomial valuation}} on $X$.  
 We say that (following \cite{JM12}) $v$ as above is adapted to the log-smooth pair $(Y,D)$. We define the center of $v$ on $Y$ to be $Z$ and the center of $v$ on $X$ to be $\mu(Z) \subset X$ for $\mu: Y \to X$ as above.
 As in \cite{JM12}, one can also define equivalence of quasimonomial valuations between one adapted to $(Y,D)$ and another adapted to a different log-smooth pair $(Y',D')$, using the Hironaka theorem.  The center of $v$ is well-defined and  coincides with that of a divisorial valuation in the case when $v$ is also a geometric divisorial valuation.

  Let  $\QM_Z (Y,D)$ denote the set of all $\tau_{Z, D, \al}$ when $(\al_1, \ldots, \al_m) \in \RR^m_{\ge 0}$ varies. We note that a quasimonomial valuation $v$ belongs to the  interior of $\QM_Z (Y,D)$  if and only if the center of $v$ on $Y$ is equal to $Z$ (cf.\ \cite[p.61]{BFJ14}).

 Define the log discrepancy $A(v)$ of a quasimonomial valuation $v \in \QM_Z (Y,D)$ as given in \eqref{DJ} by (cf. \cite[Prop.5.1]{JM12})
 
 \begin{equation}\label{Avv}
  A(v) := \tsum^m_{i=1} v(D_i) A(\ord_{D_i}) = \tsum^m_{i=1} v(D_i) (1 + \ord_{D_i} (K_{Y / X}) )
 \end{equation}
  where $v(D_i) := v(z_i) = \al_i$. (For simplicity of notation, here we denote $D_{j_i}$ by $D_i$.)  We note that $A(v)$ is linear in $\al_1, \ldots, \al_m$. This coincides with the earlier definition of $A(v)$ when $v$ also happens to be a divisorial valuation, cf. \cite{JM12}.

\begin{remark1}
 One can also consider $\QM(Y,D) := \bigcup_{Z} \QM_Z (Y,D) $ the set of all quasimonomial valuations described in terms of $(Y,D)$ given a proper modification $Y \to X$ and $D$, where $Z$ is taken over all those strata satisfying the setting of \eqref{DJ} for all $m \ge 1$. More details of $\QM(Y,D)$ are presented in \cite{JM12, B21, BFJ14} whereas we only need $\QM_Z (Y,D)$ in the proof of Theorem~\ref{converge2}.    In particular, we can identify $\QM_Z (Y,D)$ with the cone $\RR^J_{\ge 0}$ which is sitting inside the vector space $\RR^I$ where the coordinates $\al_k$ with $k \in I \setminus J$ is taken to be zero. Here we may identify $\RR^I$ with $\Div (Y,D)^{*}_{\RR}$, the dual to the $\RR$-vector space spanned by the prime components of $D$. As in \cite{B21},  one can  embed $\QM (Y,D)$ as  
 $$  \QM (Y,D) = \bigcup_{J \subset I} \RR^J_{\ge 0} \subset \RR^I \cong
 \Div (Y,D)^{*}_{\RR}. $$
 \end{remark1}

   \section{Quasimonomial valuations computing the  log canonical threshold}\label{jmjm}

From the valuative characterization of the log canonical thresholds, Corollary~\ref{char},  it is also possible to enlarge $T_x$ by the set of all quasimonomial valuations whose center contains $x$, cf.\ \cite[Thm.5.5]{BFJ08}. 
 A priori, the statement with the smaller $T_x$ is stronger, however it is also important in this paper to look for "achievements" of the infimum in the larger set $T_x$.

 \begin{definition}\label{computer1}
 Let $X$ be a complex manifold and $\vp$ a quasi-psh function on $X$. Let $T_x$ be the set of all geometric quasimonomial valuations on $X$ whose centers contain $x$. 
 We  say that a geometric quasimonomial valuation $v_0 \in T_x$ \textbf{\emph{computes}} $\lct_x (\vp)$ if the infimum is achieved by $v_0$, that is,  $ {A(v_0)}/{v_0 (\vp)} =  \inf_{v \in T_x} {A(v)}/{v(\vp)}. $

\end{definition}

 We will use the methods of Jonsson-Musta\c{t}ă~\cite{JM14} to show the following  main result of this section. 

\begin{theorem} \label{computer}

 Let $\vp$ be a quasi-psh function on a complex manifold $X$.  
Suppose that $(X, \vp)$ is log canonical  and
 let $Z \subset X$ be an irreducible component of the non-klt locus $W(\vp)$ of $(X, \vp)$. 
 Let $x \in Z$ be a  point such that $Z$ is smooth irreducible in a neighborhood of $x$.  Then there exists a geometric quasimonomial valuation $v_1$ that computes $\lct_x (\vp)$, defined on a neighborhood $U \subset X$ of $x$, such that the center of $v_1$ is $Z \cap U$. 
 
\end{theorem}

  We will use  the following lemma (which does not require log canonical). 
 
 \begin{lemma}\label{vpp} \cite[Proof of Lem.\ 4.1]{JM14} 
   Let $\vp$ be a quasi-psh function on a complex manifold $X$. Let $x \in X$ be a point. Let $V := V_x \subset X$ be the sublevel set $$V_x := \{ y \in X : \lct_y (\vp) \le \lct_x (\vp) \}.$$
Let $I_V$ be the ideal sheaf of  $V$ and define $\log \abs{I_V} := \log \sum^m_{j=1} \abs{g_j}$  in terms of (a choice of) local generators $g_1, \ldots, g_m$ of $I_V$ near $x$.  
   Define a quasi-psh function $\tilde{\vp}$ in a neighborhood $U$ of $x$ by 
$$\tilde{\vp} (z) := \max \{ \vp (z) , p \log \abs{I_V} (z) \} $$ where $p \ge 0$ is an integer. 

\begin{enumerate}
\item
 If $p \gg 0$, then we have $\lct_x (\tilde{\vp}) = \lct_x (\vp)$.  

\item For such $p$ as in (1), a geometric quasimonomial valuation $v \in T_x$ computes $\lct_x (\vp)$ if and only if it computes $\lct_x (\tilde{\vp})$. 

\end{enumerate}

\end{lemma}

\begin{proof}
 
 (1) is contained in the proof of \cite[Lem.4.1]{JM14} (in the case $\mathfrak{q} = \OO_X$). 

For (2), first note that, for each geometric quasimonomial valuation $v$, we have $v(\vp) \ge v(\tilde{\vp})$. Indeed, this follows from Demailly's  comparison theorem~\cite{D11} for Lelong numbers and Kiselman numbers.   
  This combined with  Corollary~\ref{char} yields: 
  $$\lct_x (\vp) = \inf_v \frac{A(v)}{v(\vp)} \le \inf_v \frac{A(v)}{v(\tilde{\vp})} = \lct_x (\tilde{\vp}) $$
\noi where the inf's are taken over $v$,  quasimonomial valuations whose center contains $x$. 
  This inequality turns out to be equality by  (1). 
 \end{proof}

\subsection{Use of algebraic quasimonomial valuations}  {\ }

 In this subsection,  we will complete the proof of Theorem~\ref{computer} using the methods of \cite{JM14}. Even though Theorem~\ref{computer} is stated on a complex manifold, \cite{JM14} 
uses the algebraic setting of excellent regular local domains (given by the localization of the analytic  local ring $\OO_x$ at the ideal $I_V \cdot \OO_x$). 

Let us briefly recall the  algebraic notion of a valuation (cf. \cite[Section 2.1]{JM14}, \cite{JM12}). 
 Let $K$ be a field and let $G$ be a totally ordered abelian group (e.g. $G = \ZZ$ or $\RR$). A \emph{valuation} of $K$ with values in $G$ is a function $v: K \setminus \{ 0 \} \to G$ such that for every nonzero $x,y \in K$, one has $v(xy) = v(x) + v(y)$ and $v(x+y) \ge \min (v(x), v(y))$. 
  When $G$ is isomorphic to an ordered additive subgroup of $\RR$, one says that $v$ is a real valuation (or rank $1$ valuation). 
  
  For example, $K$ can be taken to be the function field of an algebraic variety over a field. In our situation, $K$ will be the field of fractions of an excellent regular local domain $R$, so that a valuation $v$ can be taken as $v: R \setminus \{ 0 \} \to \RR_{\ge 0}$.

Before the proof of Theorem~\ref{computer}, we recall from \cite[\S 2.2]{JM12}, \cite{JM14}: 
 a sequence of nonzero ideals $\mfbb = (\mfb_j)_{j \ge 0}$ in $R$ is said to be \emph{subadditive} if $\mfb_{i+j} \subset \mfb_i \cdot \mfb_j$ for every $i, j \ge 0$, cf.\ \cite[Sec.2.1.5]{JM14}. 
     The log canonical threshold $\lct (\mfbb)$ is defined as $ \inf_{j \ge 1} j \cdot \lct(\mfb_j)$ \cite[\S 2.1.5]{JM14}. This also has the valuative characterization \cite[(2.5)]{JM14} 
  \begin{equation}\label{lctb}
   \lct (\mfbb) = \inf_v \bigl({A(v)}/{v(\mfbb)}\bigr), 
  \end{equation}
 \noi where the infimum is taken over all valuations $v$ on $R$ with $A(v) < \infty$.

  \begin{proof}[Proof of Theorem~\ref{computer}]
 
 We will use the arguments in \S 4.2 of  \cite{JM14}, which we recall for the sake of explicitness and convenience for readers.

 \emph{Step 1}.  Let $\OO_x$ be the analytic local ring of holomorphic function germs at $x \in X$, which is an excellent regular local ring (cf.\ \cite[Thm.102]{M80}). 
 Take $R = \OO_{x, Z}$ to  be the localization of $\OO_x$ at the ideal $I_Z \cdot \OO_x$, which is also an excellent regular local ring.
 
 We will later apply Lemma~\ref{vpp}. Since $(X, \vp)$ is log canonical, $\lct_y (\vp) \ge 1$ at every point $y \in X$. Also note that $\lct_x (\vp)= 1$ since $x \in Z$.  Hence we have $Z = V$ near $x$ where $V = V_x$ is the sublevel set in Lemma~\ref{vpp}. Take $\tilde{\vp} = \max (\vp, p \log \abs{I_Z})$ for $p \gg 0$. 
 
  Let $\widetilde{\mfbb}$ be the subadditive system of ideal sheaves given by $\widetilde{\mfb}_m = \JJ(m\tilde{\vp})$. 
Let $\mfbb$ denote the induced subadditive system $\widetilde{\mfbb} \cdot \OO_{x, V}$ in $R$.  
    From \cite[Prop.3.12(i)]{JM14}, we have $\lct_x (\tilde{\vp}) = \lct (\mfbb)$. 
 Then by Theorem~\ref{JMX}, there exists an algebraic quasimonomial valuation $v_0$ on $R$ that computes $\lct (\mfbb)$. 
Note that the condition $\mfm^{pj} \subset \mfb_j$ holds due to the definitions of $\tilde{\vp}$ and $R, \mfm$ and the analytic nullstellensatz~\cite[Chap 4, \S 1]{GR84}. 
\\

\emph{Step 2}. Let us show that $v_0$ can be converted into a geometric quasimonomial valuation $v_1$ (defined by the same geometric data). This is the content of \cite[Secs.3.3 and  2.2.1]{JM14}, which we recall here.

 Given $v_0$, let $\pi: X \to \Spec \OO_{x, V} $ be a projective birational morphism adapted to $v_0$ in the sense of \cite[Sec.2.1.1]{JM14}, say with an snc divisor $D_1 + \ldots + D_m$ on $X$. From its projectivity, $\pi$ is equal to the composition of a closed embedding 
 $$X  \hookrightarrow \Spec \OO_{x,V} \times_{\Spec \CC} 
 \PP^N_{\CC} $$ and then the first projection to $\Spec \OO_{x,V}$. In this embedding, $X$ is defined by a finite number of homogeneous polynomial equations with (a finite number of) coefficients in $\OO_{x, V}$, say of the form $\frac{f_j}{g}$ where $f_j \in \OO_x$ and $g \in \OO_x \setminus I_V \cdot \OO_x$. Let $U$ be an open neighborhood of $x$ where these germs $g$ and $f_j$'s are defined as actual holomorphic functions.  Let $W \subset U$ be the analytic subset given by $(g=0)$.

  One can define  the analytification of $\pi$ as a proper modification $\pi^{an} : X^{an} \to U \setminus W$ as the composition of the inclusion $X^{an} \to (U \setminus W) \times \PP^N_{\CC}$ defined by the same corresponding equations as above and the first projection.  Then as discussed in \cite[Sec.3.3]{JM14}, by shrinking $U$ and increasing $W$ if necessary, one obtains a geometric quasimonomial valuation $v_1 := \tau_{Z^{an}, D^{an}, \al}$ where $D^{an} = D^{an}_1 + \ldots + D^{an}_m$ is a suitable analytification of the original snc divisor $D_1 + \ldots + D_m$ and $Z^{an}$ is the transversal intersection of the components. 
 We have $A(v_0) = A(v_1)$ since they are defined by the same data of snc divisors. 
\\

\emph{Step 3}.
 Let us check that this $v_1$ is what we want.  Note that $v_1$  computes $\lct_x (\tilde{\vp})$ (which is equal to $\lct (\mfbb)$ from Step 1) in the sense of Definition~\ref{computer1} since \cite[Prop.3.12]{JM14} showed $v_1 (\tilde{\vp}) = v_0 ( \mfbb)$ using Demailly approximation. 
Therefore we have $$ \lct_x (\tilde{\vp}) = \lct (\mfbb) =  \frac{A(v_0)}{v_0(\mfbb)} = \frac{A(v_1)}{v_1(\tilde{\vp})}, $$ where the second equality is due to the fact that $v_0$ computes $\lct (\mfbb)$.  
 By Lemma~\ref{vpp}, $v_1$ also computes $\lct_x (\vp) = 1$.

  Now consider the center of $v_1$. Since the center of $v_0$ is given by a prime ideal in the local ring $O_{x, Z}$, the center of $v_1$ is a subvariety (say $W$) containing $Z$ near $x$. We claim that $W=Z$. Suppose that $W \neq Z$. Since $A(v_1) - v_1 (\vp) = 0$, we get contradiction from  the fact that $Z$ is an irreducible component of the non-klt locus $W(\vp)$ and thus $Z$ cannot be properly contained in another analytic subset $W \subset W(\vp)$. 
 \end{proof}

\begin{corollary}\label{algeq}

 If $X$ in Theorem~\ref{computer} is one associated to a smooth complex projective variety, then the quasimonomial valuation $v_1$ in the statement of Theorem~\ref{computer} can be taken to be an algebraic quasimonomial valuation of the function field of $X$. 
 \end{corollary}

 \begin{proof} 
 
 It suffices to follow Step 1 of the above proof of Theorem~\ref{computer} taking $\OO_x$ to be the algebraic local ring at $x \in X$ with necessary modifications.
 
  Since $X$ is projective, the sublevel set $V$ is a global algebraic subset whose reduced ideal sheaf is taken to be $I_V$. 
 In Lemma~\ref{vpp}, we can define $\log \abs{I_V}$ as a well-defined quasi-psh function on $X$ (which is equal to $\log \sum \abs{g_j}$ up to $O(1)$ in terms of local generators) using an ample line bundle $L$ such that $L \otimes I_V$ globally generated. Define a quasi-psh function as before, $\tilde{\vp} := \max \{ \vp, p \log \abs{I_V} \}$ with the same conclusion of Lemma~\ref{vpp}. 
  
   In the proof of   Theorem~\ref{computer},  take $\OO_x$ to be the algebraic local ring at $x \in X$ of the algebraic variety $X$ and take $R = \OO_{x, V}$ to be the localization taken with respect to the above $I_V$. Take $\tilde{\mfb}_m := \JJ(m \tilde{\vp})$ which is a coherent ideal sheaf of the projective variety $X$. From Theorem~\ref{JMX}, we obtain an algebraic quasimonomial valuation $v_0$ on $R$ that computes $\lct (\mfbb)$ as before. 
 Since the field of fractions of $R$ is equal to the function field of $X$, this $v_0$ is what we want. 
 \end{proof} 
 
 In the rest of this subsection, for the convenience of readers, we provide the statements from \cite{JM14} and \cite{X20} which were  used above when $R = \OO_{x, V}$.

\begin{theorem}\label{JMX} 
 Let $R$ be an excellent regular local domain of equicharacteristic zero. 
 Let $\mfbb$ be a subadditive sequence of ideals in $R$ satisfying the controlled growth condition~\cite[p.121]{JM14}, i.e. for all quasimonomial valuations $v$ on $R$ and all $j \ge 1$, one has $$
 \frac{1}{j} v(\mfb_j) \le v (\mfbb) \le \frac{1}{j} (v(\mfb_j) + A(v)) .$$
 Also suppose that there exists an integer $p \ge 1$ such that  $\mfm^{pj} \subset \mfb_j$ for all $j \ge 1$ for the maximal ideal $\mfm$ or $R$. 
 Then there exists an algebraic quasimonomial valuation $v_0$ on $R$ such that $v_0$ computes $\lct (\mfbb)$.

\end{theorem}

 \begin{proof}
 
 This is contained in Conjecture E of \cite{JM14}. By \cite[Prop.2.3]{JM14}, this follows from Conjecture C'' of \cite{JM14}, proved by  Xu~\cite{X20}; see  Theorem~\ref{Xu}.  
 \end{proof}

 Given a  graded system of ideals $\mfab$, an algebraic version of log canonical threshold $\lct (\mfab)$
 is defined as   $ \lct (\mfab) := \ds \inf_v\bigl( {A(v)}/{v(\mfab)}\bigr) $ similarly to the case of a subadditive system \eqref{lctb}. (See \cite{JM12} for more details on $v(\mfab)$.) 
  
 \begin{theorem}\label{Xu} \cite[Thm.1.1]{X20}
 Let $X$ be a smooth complex algebraic variety and let $\mfab$ be a graded system of ideals on $X$ with $\lct (\mfab) < \infty$. Then  there exists a quasimonomial valuation that computes $\lct (\mfab)$. \qed
 
 \end{theorem}
 
 In fact, in \cite{X20}, this was proved more generally for a klt pair $(X, \Delta)$.

\section{Almost log canonical centers}\label{sec.6.jk}

The aim of this section is to show that once a log canonical pair
$(X, \Delta)$ is  `close to being non-klt' along a Zariski closed subset $W\subset X$, we can find another  log canonical pair
$(X, \Delta^*)$ that is not klt along $W$; see
Theorem~\ref{near.lc.=.lc.q.thm}. To state it,  we need some definitions first to clarify the meaning of `close to being non-klt.'
The right notion is given by the minimal log discrepancy  (Definition~\ref{mld}) and the 
log canonical gap.

\begin{defn}\label{epsilon.defn}
  The \textbf{\emph{log canonical gap}} in dimension $n$, denoted by $\lcg(n)$, 
    is the smallest real number  $\epsilon$
  with the following property.
  \begin{enumerate}
\item Let  $(X, \sum d_iD_i)$ be a  $\q$-factorial, lc pair with $\dim X=n$. Assume that 
$0\leq 1-d_i<\epsilon$ for every $i$. Then $(X, \sum D_i)$ is also lc.
  \end{enumerate}
\end{defn} 
  It is easy to see that $\lcg(2)=\frac16$ and $\lcg(3)=\frac1{42}$ by \cite[5.5.7]{Ko94}.
  A  difficult theorem  \cite[Thm.1.1]{HMX14} says that $\lcg(n)$  is positive  for every $n$. No explicit lower bound is known, but
  $\lcg(n)$ converges to 0 very rapidly  \cite[8.16]{Ko97}.

  \smallskip
      {\it Comments on the references.}
      Even though we will use Theorem~\ref{near.lc.=.lc.q.thm} for a complex manifold $X$ in Theorem~\ref{main}, in the proof    we need to use $\lcg(n)>0$ for a  proper modification $Y\to X$, and $Y$ is usually quite singular. So we do need the full generality of \cite[Thm.1.1]{HMX14}, the smooth version  \cite{DEM10}  would not be enough. 

      Note also that \cite{HMX14} works with algebraic varieties.
      However, the proof in \cite[Sec.5]{HMX14} reduces the claim to a similar assertion about the fibers of a dlt modification. Thus, once we know that relatively projective  dlt modifications exist by \cite{F2201, LM22}, the proof in \cite{HMX14} works for analytic spaces as well.

\begin{assumption}\label{f.t.ass} In Theorem~\ref{near.lc.=.lc.q.thm} we work in one of the following settings.
     \begin{enumerate}
     \item  Quasi-projective varieties (say $X$) over a field of characterisic 0.
     \item  Schemes (say $X$) that are quasi-projective over a quasi-excellent, affine $\q$-scheme.
       \item   Complex analytic spaces (say $X$) that are projective over a  Stein space $S$ (with $\pi: X \to S$). 
       \end{enumerate}
The main reason for these restrictions is that we need line bundles with general enough sections. In particular, if $X$ satisfies one of (1--3) then the following holds.
 \begin{enumerate}\setcounter{enumi}{3}
 \item Let $g:Y\to X$ be a proper morphism and $L_Y$ a line bundle on $Y$, such that, every $x\in X$ has an open neighborhood  $x\in U\subset X$ for which $L_Y|_{g^{-1}(U)}$ is globally generated. Then, the following hold. 
 
 \begin{itemize}
 
 \item In cases (1--2)  there is a  line bundle $L_X$ on $X$ such that  $L_Y\otimes g^*L_X$  is globally generated.
 
 \item
   In case (3),  there is a holomorphic line bundle $L_X$ on $X$ such that, every compact subset $C\subset S$ has a
   Stein open neighborhood $C\subset S^\circ \subset S$
satisfying that  $L_Y \otimes g^*L_X$  is globally generated over $(\pi \circ g)^{-1} ( S^\circ)$.
  
  \end{itemize}

    \end{enumerate}
    
Throughout the proof of Theorem~\ref{near.lc.=.lc.q.thm}, we use  several results about singularities of algebraic varieties and minimal models, whose proofs originally used global methods.
Many of these need generalizations of Kodaira's vanishing theorem.

Proofs of the required vanishing theorems are given in \cite{T85} for analytic spaces and in \cite{Mu21} for quasi-excellent schemes.
 \cite{F2201} 
establishes  the Minimal Model Program over complex analytic  bases, and 
\cite{LM22}  over  quasi-excellent base schemes.
Once the basics are established, the theorems of the algebraic MMP work in the complex analytic setting with 2 caveats.
   \begin{enumerate}\setcounter{enumi}{4}
   \item We work with projective morphisms  $\pi:X\to B$ where $B$ is Stein.
   \item  We fix a compact subset $C\subset B$, and the conclusions apply not to the whole $X$,  but to  $\pi^\circ:X^\circ:=\pi^{-1}(B)\to B^\circ$ for some  Stein open subset $B^\circ$ such that $ C\subset B^\circ \subset B$.
   \end{enumerate}
   Note that (5) is quite important : there are serious issues with MMP even for proper (but not projective) morphisms between varieties.
   By contrast one can  avoid (6) in many cases. \cite{F2201}  works essentially with the germ of $B$ around $C$; this requires some easy to satisfy mildness assumptions on $C$, see \cite[p.3]{F2201}. 
Otherwise, it is sometimes possible to work with an increasing sequence of $C_i\subset B$ and pass to the limit. However, Theorem~\ref{near.lc.=.lc.q.thm} involves choices, and it is not clear that working over the whole base is possible.

   Thus we state Theorem~\ref{near.lc.=.lc.q.thm} precisely, but in the auxiliary Lemmas~\ref{mmp.on.crep.gen.lem}--\ref{complem.model.prop} we ignore passing to $X^\circ\subset X$, since this would make the notation very cumbersome.
\end{assumption}

\begin{thm}\label{near.lc.=.lc.q.thm}
  Let $(X, \Delta)$ be a klt pair  of dimension $n$ with $X$ satisfying one of the Assumptions~(\ref{f.t.ass}.1--3).
   Let $W\subset X$ be a closed subset (a closed analytic subset in (\ref{f.t.ass}.3))  such that   $\mld(W, X,\Delta)<\lcg(n)$.
 Then the following hold in the respective cases. 
\begin{enumerate}
\item  In the algebraic cases  (\ref{f.t.ass}.1--2), there is an  lc  pair
  $(X, \Delta^*)$ such that  $W$ is equal to the union of all lc centers of  $(X, \Delta^*)$.
  \item In the analytic case  (\ref{f.t.ass}.3) with $\pi: X \to S$, for every compact subset $C\subset S$, there exist
  
 \begin{itemize}
 \item
   a Stein space $S_C$ such that  $C\subset S_C \subset S$, and
   
   \item 
   an  lc  pair
  $(X_C, \Delta^*_C)$ where $X_C := \pi^{-1} (S_C)$  such that the union of all lc centers of  $(X_C, \Delta^*_C)$ is equal to $W\cap X_C$. 
  \end{itemize}
  \end{enumerate}
\end{thm}

\begin{rem}
 Usually one cannot choose $\Delta^*$ (or $\Delta^*_C$) to satisfy  $\Delta^*\geq \Delta$, as shown by the 2-dimensional example
 $\bigl(\CC^2, \Delta:=(1-\eta)(x=0)+(1-\eta)(y=0)+\eta(x=y)\bigr)$.

 If we increase the coeffcients $1-\eta $ to $1$, we must eliminate the
 $\eta(x=y) $ summand. Thus, in the neighborhood of the origin, the only choice is
 $\Delta^*=(x=0)+(y=0)$. So $\Delta^*$ is locally unique.

 In most other cases, there are many choices for $\Delta^*$. It is not clear whether  some  are better than  others.
  \end{rem}

We start the proof of Theorem~\ref{near.lc.=.lc.q.thm} with some preliminary lemmas.

\begin{lem}\label{mmp.on.crep.gen.lem} Let $Z$ be as in  Assumption~\ref{f.t.ass} (i.e. taking its $X :=Z$) and 
   $f:Y\to Z$  a projective, bimeromorphic morphism (i.e. a modification). Let $D$ be a (not necessarily effective) divisor. Assume that the $D$-MMP over $Z$ runs and terminates.  
  Then, at the end we get  $f^m:Y^m\to Z$ such that 
  \begin{enumerate}
    \item $D^m$ is $f^m$-nef, 
    \item every step of the MMP is an isomorphism over $Z\setminus f(\supp D)$, and 
     \item $f(\supp D_{(<0)})=f^m(\supp D^m_{(<0)})$.
     \end{enumerate}
\end{lem}

\begin{proof} The first   claim holds by definition of the MMP and (2) is clear since we only contract curves that have negative intersection with the  bimeromorphic transform of $D$.

  Next we show that at every intermediate step of the MMP, we have $f(\supp D_{(<0)})=f_i(\supp (D_i)_{(<0)})$.  This is clear for $i=0$.  As we go
  from $i$ to $i+1$, the image $f_i(\supp (D_i)_{(<0)})$ is unchanged if $Y^i\map
  Y^{i+1}$ is a flip.  Thus let $\phi_i:Y_i\to Y_{i+1}$ be a divisorial contraction
  with exceptional divisor $E_i$ and let $F_i\subset E_i$ be a general fiber of
  $E_i$  (over its image in $Y$).  It is clear that
  $f_{i+1}(\supp (D_{i+1})_{(<0)})\subset f_i(\supp (D_i)_{(<0)})$. 
  Note that  $(D_i)_{(>0)}|_{F_i}$ is effective and $-E_i|_{F_i}$ is nef. On the other hand,  $D_i|_{F_i}$ is negative. Thus $(D_i)_{(<0)}$ has  another irreducible component that intersects $F_i$, so $f_{i+1}(\supp (D_{i+1})_{(<0)})\supset f_i(E_i)$.
\end{proof}

The following is a variant of \cite[Lem10]{KK22}.

\begin{lem}\label{mmp.D.r.t.lem} Let $Z$ be as in Assumption~\ref{f.t.ass} (i.e. taking its $X :=Z$)  and 
   $f:Y\to Z$  a projective,  bimeromorphic morphism.
   Assume that  $(Y, \Delta)$ is $\q$-factorial, klt and $K_Y+\Delta\sim_{f,\RR}0$.
  Let $D$ be a (not necessarily effective) divisor such that $ \supp D\subset \supp \Delta$.

  Then   $(Y, \Delta+\epsilon D)$ is klt for $|\epsilon|\ll 1$, and the
  $D$-MMP over $Z$ runs and terminates.
\end{lem}

\begin{proof} Let $p:Y'\to Y$ be a log resolution of $(Y, \Delta)$.
  Then $p$ is also a log resolution of $(Y, \Delta+\epsilon D)$ and the discrepancies vary linearly with $\epsilon$. 
  Thus $(Y, \Delta+\epsilon D)$ is klt  for $|\epsilon|\ll 1$.

  Since  $K_Y+\Delta\sim_{f,\RR}0$, the $D$-MMP over $Z$ is the same as the
  $K_Y+ \Delta+\epsilon D$-MMP over $Z$. 
  \end{proof}

\begin{lem}\label{complem.model.prop} Let $Z$ be as in Assumption~\ref{f.t.ass} (i.e. taking its $X :=Z$)  and 
   $f:Y\to Z$  a projective,  bimeromorphic morphism.
 Assume that  $(Y, \Theta+\Delta)$ is $\q$-factorial, klt and $K_Y+\Theta+\Delta\sim_{f,\RR}0$. Assume also that $\coeff\Theta\subset (1-\lcg(n), 1)$.

  Then,  after a suitable MMP we get  $f^m:\bigl(Y^m, \Theta^m+\Delta^m\bigr)\to Z$ and an $\RR$-divisor $cH^m$ with rational $0<c<1$ such that 
  \begin{enumerate}
  \item $f(\supp \Theta)=f^m(\supp \Theta^m)$,
    \item $K_{Y^m}+ \rup{\Theta^m}+cH^m\sim_{f^m,\q}0$, and 
  \item $\bigl(Y^m, \rup{\Theta^m}+cH^m\bigr)$ is lc and  klt outside
   $\supp \rup{\Theta^m}$. 
  \end{enumerate}
\end{lem}

\begin{proof}  By \cite[11.38]{Ko23} we can approximate
  the $\RR$-divisor $\Theta+\Delta$ by a $\q$-divisor 
  such that all assumptions continue to hold. Thus we may as well assume that
   $\Theta+\Delta$ is a $\q$-divisor.

  We apply Lemma~\ref{mmp.D.r.t.lem} to  $D:=\Theta+\Delta-\rup{\Theta}$ to get
   $f^m:(Y^m, \Theta^m+\Delta^m)\to Z$. Note that $\supp D_{(<0)}=\supp\Theta$, hence we have (1).

  $K_{Y^m}+ \Theta^m+\Delta^m+\epsilon D^m$ is klt and $f^m$-nef, hence relatively basepoint-free. Since 
$K_{Y^m}+ \Theta^m+\Delta^m+\epsilon D^m\sim_{f^m,\q} \epsilon D^m$, 
 there is a  relatively basepoint-free linear system
 $|H^m|$ and $0<c<1$ such that  $D^m\sim_{f^m,\q} c|H^m|$.
 We choose a general divisor $H^m\in |H^m|$.
  (In order to do this, we may need to add to $|H^m|$ the pull-back of a sufficiently ample linear system first. This is  where we need (\ref{f.t.ass}.4).)
 Next note that
 $$
 K_{Y^m}+ \rup{\Theta^m}+cH^m\sim_{f^m,\RR} 
K_{Y^m}+ \Theta^m+\Delta^m-D^m+cH^m\sim_{f^m,\RR} 0,
$$
hence  (2) holds.

  Note that $\bigl(Y^m, \rup{\Theta^m}\bigr)$ is lc  by Definition~\ref{epsilon.defn}, and
  so (3) holds by the general choice of  $H^m$ (cf.\ \cite[1.13]{Re80}).
  \end{proof}

\begin{proof}[Proof of Theorem~\ref{near.lc.=.lc.q.thm}]\label{near.lc.=.lc.q.thm.pf} 
  By \cite[1.38]{Ko13}, 
    there is a  
  $\q$-factorial 
modification   $f:\bigl(Y,\Theta+\Delta_Y\bigr)\to (X,\Delta)$
such that
\begin{enumerate}
  \item $K_Y+\Theta+\Delta_Y\sim_{f,\RR} 0$, 
  \item $f_*(\Theta+\Delta_Y)=\Delta$,
\item  $\coeff\Theta\subset (1-\lcg(n), 1)$, and
\item  $f(\supp\Theta)=W$.
\end{enumerate}
Note that $\bigl(Y,\Theta+\Delta_Y\bigr)$ is klt since $(X,\Delta)$ is, by  \cite[2.30]{KM}.
We can now apply Lemma~\ref{complem.model.prop} to get
$f^m:\bigl(Y^m, \rup{\Theta^m}+cH^m\bigr)\to X$.

Note that $K_X+ f^m_*(\rup{\Theta^m}+cH^m)$ is $\q$-Cartier 
since $K_{Y^m}+ \rup{\Theta^m}+cH^m\sim_{f^m,\q}0$ by (\ref{complem.model.prop}.2). For the same reason,
$$
K_{Y^m}+ \rup{\Theta^m}+cH^m\sim_{\q}
(f^m)^*\bigl(K_X+ f^m_*(\rup{\Theta^m}+cH^m)\bigr).
$$
 Thus 
 $\bigl(X, f^m_*(\rup{\Theta^m}+cH^m)\bigr)$ is lc and
 the union of its lc centers is $f^m(\supp \Theta^m)$, again
 by  \cite[2.30]{KM}.
 Finally $W=f(\supp \Theta)=f^m(\supp \Theta^m)$ by (4) and
 (\ref{complem.model.prop}.1). We  take $\Delta^*$ (and $\Delta^*_C$) to be $f^m_*(\rup{\Theta^m}+cH^m)$. 
\end{proof}

 \section{Proofs of the main results}

 In the complex analytic setting, a prime divisor (i.e. a geometric divisorial valuation) over an open subset need not globalize, but the following result shows that
 this is not a problem for divisors with small log discrepancy.

\begin{proposition}\label{komo} Let $(X, \Delta)$ be a klt pair, with
$\Delta$ a not necessarily effective $\RR$-divisor.
    Let
    $$
    X_r\stackrel{\pi_r}{\longrightarrow} X_{r-1}
    \stackrel{\pi_{r-1}}{\longrightarrow}\cdots
    \stackrel{\pi_1}{\longrightarrow} X_0=X
    $$
    be a sequence of morphisms. Assume that every $\pi_i$ is either proper
and bimeromorphic,  or an open embedding.
    By pull-back and restriction, we get klt pairs  $(X_i,\Delta_i)$.

    Let $D_r$ be a prime divisor over $X_r$ such that $A(D_r,
X_r,\Delta_r)<1$. Then there is a geometric divisorial valuation $D$ over $X$
    such that $D$ and $D_r$ induce  equivalent geometric divisorial valuations over
$X_r$.
    \end{proposition}

    \begin{proof} By induction, it is sufficient to prove this for
      a single step $\pi_i:X_i\to X_{i-1}$ of the chain. The claim is clear
if $\pi_i$ is proper. If $\pi_i$ is an open embedding, let us denote it by
$j:U\hookrightarrow X$.

      Take a log resolution $g_X: (X', \Delta')\to (X,\Delta)$
      such that the negative part of $\Delta'$ is smooth. That is,
      $\Delta'=\Delta'_1-\Delta'_2$ where the $\Delta'_i$ are effective and
      $\Delta'_2$  is smooth.  For such a pair, every
  divisorial valuation $D'$ satisfying
  $A(D', X',\Delta')<1$ is equivalent to an irreducible component of
   $\Delta'_2$ by \cite[2.7]{Ko13}. 

      By  restriction we  get a
      log resolution $g_U: (U', \Delta'_U)\to (U,\Delta|_U)$.
As before, every
  divisorial valuation $D'_U$ satisfying
  $A(D'_U, U',\Delta'_U)<1$ is equivalent to an irreducible component of
   $\Delta'_2|_U$, hence to an irreducible component of
   $\Delta'_2$.
  \end{proof}

   \begin{remark1} The bound $A(D_U, U,\Delta|_U)<1$ is sharp.
     Indeed, if $\Delta|_U$ contains a divisor $B_U$ with coefficient 1, and
     $Z_U\subset B_U$ has codimension 1, then blowing up $Z_U$ produces a divisor $D_{Z_U}$ with log discrepancy 1. This $D_{Z_U}$ is equivalent to a geometric divisorial valuation over $X$ iff the closure of $Z_U$ is an analytic subset.

     However, if $(X, \Delta)$ is klt, then the claim holds for
     $A(D_U, U,\Delta|_U)\leq 1$, which can be further improved depending on the total discrepancy of $(X, \Delta)$; see 
     \cite[2.7]{Ko13}.
     \end{remark1}

  \subsection{Proofs of Theorem~\ref{logdis} and Corollary~\ref{four}}  {\ }

 Using quasimonomial valuations from Section 5, we prove Theorem~\ref{logdis}.

 \begin{theorem}[=Theorem~\ref{logdis}]\label{converge2}
 
Suppose that $(X, \vp)$ is lc  and
 let $Z \subset X$ be an irreducible component of the non-klt locus of $(X, \vp)$. Then there exists a sequence of geometric divisorial valuations $G_j$  over $X$ whose centers are equal to $Z$,   such that the log discrepancies $A(G_j, X, \vp)$ converge to $0$. 
 \end{theorem}

\begin{proof}  
First we construct these $G_j$ over a   neighborhood of a nonsingular point of $Z$.

   Let $x \in Z$ be a nonsingular point such that other irreducible components of the non-klt locus do not contain $x$. Since $(X, \vp)$ is lc, this $Z$ is the same as the sublevel set $V$ in Lemma~\ref{vpp}, $\{ y \in X : \lct_y (\vp) \le \lct_x (\vp) = 1 \}$, in a neighborhood of $x$. 
  By Theorem~\ref{computer}, there exists a geometric quasimonomial valuation $v$ that computes the $\lct_x (\vp)$, defined on a neighborhood  $U$ of $x$.  The center of $v$ is equal to $Z \cap U$.  
  
 Let $\mu: Y \to U$ be a proper modification such that $v$ is adapted to a log-smooth pair $(Y,D)$. 
  We may assume that (cf. \cite[\S 3]{JM12})  $v$ belongs to the relative interior of some $\QM_{T} (Y,D)$ for a stratum $T$, i.e. a connected transversal intersection of components of $D$ as in \eqref{DJ}.

  At this point, we  use two ingredients from the argument in \cite[Lem.3.1]{LXZ21}, which are originally from \cite{LX18}. 
  
  (1) 
   Due to linearity of $A(v)$ (from its definition)  and concavity of $v(\vp)$ (cf.\ \cite[Props.5.6 and 3.11]{BFJ08}, both as functions on the valuations $v$ in $\QM_{T} (Y,D)$, the function $A(v, X, \vp) := A(v) - v(\vp)$ is convex and hence locally Lipschitz continuous~\cite[Thm. 2.1.22]{Ho94}. Therefore, there exists a compact neighborhood $K$ of $v$ in  $\QM_{T} (Y,D)$   and constants $C = C_K > 0, \ep_0 > 0$ such that 
   \begin{equation}\label{Lip}
   \abs{A(w, X, \vp) - A(v, X, \vp)} \le C \abs{w-v}  
   \end{equation}
for every $w \in K$. 

 (2) From the diophantine approximation lemma of \cite[Lem.2.7]{LX18}, 
  for any $\ep > 0$, there exist divisorial valuations $$v_1 = \tfrac{1}{q_1} \ord_{F_1}, \ldots, v_r = \tfrac{1}{q_r} \ord_{F_r}$$ in $\QM_{T} (Y,D)$, where  $F_1, \ldots, F_r$ are prime divisors  over $X$ and   $q_1, \ldots, q_r$ are positive integers. Also we may assume that $v$ is in the convex cone generated by $v_1, \ldots, v_r$, and  $\abs{v_j - v} < \tfrac{\ep}{q_j} $, where the Euclidean distance is taken in the cone $\QM_{T} (Y,D) \cong \RR^r_{\ge 0}$.

  Now apply \eqref{Lip} to $w = v_i$, which certainly belongs to $K$.  Since $A(v, X, \vp) = 0$, we have, for every $i$, 
    $$ \abs{A(F_i) - F_i (\vp) - 0 } = q_i \abs{A(v_i) - v_i (\vp) } \le C q_i \abs{v_i - v} \le C \ep.$$
For every $j$, taking $\ep = \frac{1}{j}$ and $G_j$ to be one of $F_i$'s completes the proof, noting that the center of $F_i$ is equal to $Z \cap U$,  as it belongs to the same relative interior of $\QM_T (Y,D)$. 

Applying  Proposition~\ref{komo} shows that these $G_j$ are 
equivalent to geometric divisorial valuations $G_j^X$ over $X$. 
\end{proof}

Now we give the proof of Corollary~\ref{four}. Note that the proof itself uses Definition~\ref{lccenter} instead of Theorem~\ref{logdis}. In the case when $Z$ is a maximal log canonical center, i.e. an irreducible component of the non-klt locus, it is from Theorem~\ref{logdis} that the condition of Definition~\ref{lccenter} holds for $Z$.

\begin{proof}[Proof of Corollary~\ref{four}]

Since $\vp_m \ge \vp + O(1)$ for every $m$, the pair $(X, \vp_m)$ is lc. Hence $\mld (Z, X, \vp_m) \ge 0$. 

Now from Definition~\ref{lccenter}, there exists a sequence of prime divisors over $X$ with its center equal to $Z$, $\{ E_i \}_{i \ge 1}$ such that $A(E_i , X, \vp) \to 0$ as $i \to \infty$. Therefore,  given $\ep > 0$, there exists $i \ge 1$ such that 
 $ \big| A(E_i, X, \vp) \big|  < \tfrac{1}{2} \ep. $ Also by Lemma~\ref{cmvpm}, there exists $m \ge 1$ (which may depend on $E_i$) such that 
 $$\big| A(E_i, X,  \vp_m) - A (E_i, X, \vp) \big|  < \tfrac{1}{2} \ep .$$  
From Definition~\ref{mld} of minimal log discrepancy,  we have $$\mld (Z, X, \vp_m) := \inf_E A(E, X,  \vp_m)$$ where the infimum is taken over $E$'s whose center is equal to $Z$. Taking $E = E_i$ from the above completes the proof. 
 \end{proof}

\begin{lemma}\label{cmvpm}
Let $E$ be a prime divisor over (i.e. a geometric divisorial valuation on) $X$. We have $ \lim_{m \to \infty} A(E, X,  \vp_m) \to A(E, X, \vp)$.

\end{lemma}

\begin{proof}

Let $v = \ord_E$. 
  By definition, $A(E, X,  \vp_m) = A(v)  - v (\vp_m)$. 
  On the other hand, we have the relation $v (\vp) \ge v (\vp_m) \ge v (\vp) - \tfrac{1}{m} A(v)$ from 
  \cite[Lem.B.4]{BBJ21}  where we take $U= m\vp$.  (This generalizes the case of $v = \ord_E$ when $E$ is the exceptional divisor of a point blow-up from \cite{D92}.)
\end{proof}

\subsection{Proof of Theorem~\ref{main}} {\ }

Based on preparations from the previous sections, we complete the proofs of the main results.
 
 \begin{proof}[Proof of Theorem~\ref{main}]
  Let  $W(\vp):=\spec_X \OO_X/\JJ(\vp)$ be the non-klt locus  with irreducible components  $W_i\subset W(\vp)$.
   Let $(\vp_m)_{(m \ge 1)}$  be the Demailly approximation sequence of $\vp$ (depending on some choice of a partition of unity on $X$).  Since each $Z := W_i$ is a maximal lc center, we can apply Corollary~\ref{four}  for $Z = W_i$ for each $i$. Hence for $m$ sufficiently large,  there are geometric divisorial valuations  $E_i$ with center $W_i$ such that
   $A(E_i, X, \vp_m) < \frac{1}{2} \lcg (n)$ for every $i$, in view of the definition of $\mld$. (Here $\lcg(n) > 0$ is the log canonical gap as in Definition~\ref{epsilon.defn}.) 
   
   We will convert this last condition of $E_i$'s into another  condition for mld with respect to a local pair $(U, \Delta_U)$ so that we can apply  Theorem~\ref{near.lc.=.lc.q.thm} (2), the consequence of the MMP over complex analytic spaces.

   Since $(X, \vp)$ is lc and $\vp_m \ge \vp + O(1)$,  we know that $(X, (1-\ep) \vp_m)$ is klt for every $\ep > 0$. 
Since $A(E_i, X, t\vp)$ is a continuous function of $t$, we can choose
$m\gg 1$ and  $0<\ep\ll 1$ such that
$
A\bigl(E_i, X, (1-\epsilon)\vp_m\bigr) <  \lcg (n)$ for every $i$.

Now given arbitrary point $p$ of $X$, let $W, U$ with $p \in W \Subset  U \Subset X$ be relatively compact Stein open neighborhoods. It suffices to consider those $p \in W(\vp)$ since otherwise one can take $\psi = 0 = 1 \log \abs{1}$.

Since $(1-\ep) \vp_m$ has analytic singularities, by \cite[p.249]{Ko13}, there is an $\RR$-divisor  $\Delta_U \ge 0$ on $U\subset X$ such that
$(U, \Delta_U)$ is klt and
$$
A(E_i, U,\Delta_U)=
A\bigl(E_i, X, (1-\epsilon)\vp_m\bigr) <  \lcg (n)\quad\mbox{for every $i$.}
$$

Using this, now we can apply  Theorem~\ref{near.lc.=.lc.q.thm} (2) taking $\pi: X= U \to S=U$ to be the identity map, $C = \overline{W}$ and taking $(X, \Delta) := (U, \Delta_U)$.   As the result, we obtain  a Stein manifold $X_C$ with $W \subset X_C \subset X$ and an lc pair $(X_C, \Delta^*_C)$
such that  $W(\varphi) \cap X_C$ is equal to the union of all lc centers of $(X_C, \Delta^*_C)$.

 Finally take $V := X_C$ and $\psi$ to be a quasi-psh function having the same algebraic poles as the divisor $\Delta^*_C$ (as in Example~\ref{2.7.exmp.jk}). 
\end{proof} 

\begin{proof}[Proof of Corollary~\ref{semi}]
  Apply  Theorem~\ref{Ambro2003}---the analytic version of the Ambro-Fujino seminormality theorem---to the new pair $(X_C, \Delta^*_C)$
  constructed in the above proof of Theorem~\ref{main}. 
\end{proof}

\begin{remark1}
Let $(X, \vp)$ be an lc pair where X is a complex manifold and $\vp$ a quasi-psh
function. Our arguments also show that if
  $W$ is a  locally finite union of log canonical centers of $(X, \vp)$, then $W$ is
seminormal.

\end{remark1}

\footnotesize

\bibliographystyle{amsplain}

\qa

\normalsize

\noi \textsc{Dano Kim}

\noi Department of Mathematical Sciences and Research Institute of Mathematics

\noi Seoul National University, 08826  Seoul, Korea

\noi Email address: kimdano@snu.ac.kr

\qa

\normalsize

\noi \textsc{János Kollár}

\noi Department of Mathematics

\noi Princeton University, Fine Hall, Washington Road, Princeton, NJ 08544-1000, USA

\noi Email address: kollar@math.princeton.edu

\end{document}